\documentclass{article}
\usepackage{amsmath,amssymb,theorem}
\usepackage{graphicx}
\usepackage{verbatim}

\usepackage{color}
\usepackage{makeidx}

\theorembodyfont{\upshape}

\def\bino{\text{Bin}}

 \makeatletter
 \@addtoreset{equation}{section}
 \makeatother

 \newtheorem{ittheorem}{Theorem}
 \newtheorem{itlemma}{Lemma}
 \newtheorem{itproposition}{Proposition}
 \newtheorem{itdefinition}{Definition}

 \newtheorem{itremark}{Remark}
 \newtheorem{itclaim}{Claim}
 \newtheorem{itcorollary}{\bf Corollary}

 \newenvironment{theorem}{\addtocounter{equation}{1}
 \begin{ittheorem}}{\end{ittheorem}}

 \newenvironment{lemma}{\addtocounter{equation}{1}
 \begin{itlemma}}{\end{itlemma}}

 \newenvironment{proposition}{\addtocounter{equation}{1}
 \begin{itproposition}}{\end{itproposition}}

 \newenvironment{definition}{\addtocounter{equation}{1}
 \begin{itdefinition}}{\end{itdefinition}}

 \newenvironment{remark}{\addtocounter{equation}{1}
 \begin{itremark}}{\end{itremark}}

 \newenvironment{claim}{\addtocounter{equation}{1}
 \begin{itclaim}}{\end{itclaim}}

 \newenvironment{proof}{\noindent {\bf Proof.\,}
 }{\hspace*{\fill}$\qed$\medskip}

 \newenvironment{corollary}{\addtocounter{equation}{1}
 \begin{itcorollary}}{\end{itcorollary}}

 \newcommand{\be}[1]{\begin{eqnarray*}\label{#1}}
 \newcommand{\ee}{\end{eqnarray*}}

 \newcommand{\bl}[1]{\begin{lemma}\label{#1}}
 \newcommand{\el}{\end{lemma}}

 \newcommand{\br}[1]{\begin{remark}\label{#1}}
 \newcommand{\er}{\end{remark}}

 \newcommand{\bt}[1]{\begin{theorem}\label{#1}}
 \newcommand{\et}{\end{theorem}}

 \newcommand{\bd}[1]{\begin{definition}\label{#1}}
 \newcommand{\ed}{\end{definition}}

 \newcommand{\bcl}[1]{\begin{claim}\label{#1}}
 \newcommand{\ecl}{\end{claim}}

 \newcommand{\bp}[1]{\begin{proposition}\label{#1}}
 \newcommand{\ep}{\end{proposition}}

 \newcommand{\bc}[1]{\begin{corollary}\label{#1}}
 \newcommand{\ec}{\end{corollary}}

 \newcommand{\bpr}{\begin{proof}}
 \newcommand{\epr}{\end{proof}}

 \newcommand{\bi}{\begin{itemize}}
 \newcommand{\ei}{\end{itemize}}

 \newcommand{\ben}{\begin{enumerate}}
 \newcommand{\een}{\end{enumerate}}

\newcommand{\tatop}[2]{\genfrac{}{}{0pt}{1}{#1}{#2}}

\def\uro{\smash{{U}^{\!\!\!\!\raise5pt\hbox{$\scriptstyle o$}}}}

\def\bp{{\overline{p}}}

\def\bp{{\overline{p}}}

 \def \ba {\begin{array}}
 \def \ea {\end{array}}

 \def \qed {{\heartsuit\hfill}}

 \def \N {{\mathbb N}}

\def \qed {{\square\hfill}}

\def \qed {{\square\hfill}}

\def\N{{\mathbb N}}

\def\eqref#1{(\ref{#1})}

\makeindex

\begin{document}
\allowdisplaybreaks[4]

\title{A basic model of mutations}

\author{ Maxime Berger
\footnote{
\noindent
D\'epartement de math\'ematiques et applications, Ecole Normale Sup\'erieure,
CNRS, PSL Research University, 75005 Paris}
\hskip 70pt 
Rapha\"el Cerf 
\footnotemark[1]
\footnote{
Laboratoire de Math\'ematiques d'Orsay, Universit\'e Paris-Sud, CNRS, Universit\'e
Paris--Saclay, 91405 Orsay.}
}

\maketitle



\begin{abstract}
\noindent
We study a basic model for mutations. We derive exact formulae for the mean time needed to discover
	the master sequence, the mean returning
time to the initial state, or to any Hamming class. These last two formulae are
	the same than the formulae
	obtained by Mark Kac for the Ehrenfest model.
\end{abstract}



\section{\bf Introduction}
According to the Darwinian paradigm,
the evolution of living creatures is driven by two main forces: mutations
and selection. Mutations create
new forms of behaviour or new characters, 
some less fit to their environment, 
some more, 
whereas  selection favors the reproduction of fitter individuals.
On the one hand, mutations may discover very fit characters but without selection,
they would be quickly erased by further mutations.
On the other hand, selection alone would result in uniform populations, lacking in genetic diversity.
The success of an evolutionary process rests on a subtle interaction between mutations
and selection.

Let us consider for instance
a population of HIV viruses or {\itshape Drosophila melanogaster}. 
The genetic material of one individual, also called its genotype, is encoded into its DNA, which is
a long chain of nucleobases A,T,G or C.
To simplify the analysis, we suppose here that there are only two types of nucleobases
instead of four, and we denote them by $0$ or $1$.
Selection enters the game through fitness.
The fitness describes the adaptation of the individuals to the environnement.
The fitness of an individual can be thought as a function of its genotype.
For instance, a possible choice for the fitness function is the expected number of offspring 
of the individual.
We consider the situation where all the genotypes are equally fitted, except one,
say $0\cdots 0$, which has a superior fitness.
There exist several 
mathematical models of evolution combining mutations and selection.
The simplest one 
is perhaps the Moran model, whose dynamics is the following.
At each time step, one individual dies, while one individual gives birth to a child 
(in particular, the population size stays constant).
All individuals are equally likely to die, but the fitter individuals having
genotype $0\cdots 0$ reproduce more often.
Mutations occur during reproduction: the genotype of the child is not an exact
copy of the one of its parent.

A central question is then to determine the proportion of individuals 
with genotype $0\cdots 0$
in the population after a long time. 
To answer this question, one should understand how long it takes for 
a population to escape from
the set of the selectively neutral genotypes. We perform this crucial step here.
We consider one single individual, we follow his lineage and we 
compute the time needed for the genotype $0\cdots 0$ to be discovered.

\noindent
{\bf A model for mutations}.
We follow the genotypes of the individuals along a lineage.
We fix the probability of mutation $p\in (0,1)$.
We start from a string $X_0$ in $\{0,1\}^N$, which represents the genotype of the
first individual.  
We denote by $X_n$ the genotype of the $n$--th individual in the lineage.
At time $n$, for each bit of $X_n$, we toss a coin of parameter~$p$ to decide whether a
mutation occurs on the bit, in which case it is transformed into the complementary digit.
All the coins are taken independent, so the probability of having no mutations at all at time $n$ is $(1-p)^n$.
The assumption of independance simplifies considerably the mathematical analysis, and it is also
biologically plausible. Indeed the mutations arise because of transcription errors during the replication
process, and for long genomes, they are not correlated. It seems however that the mutation probability $p$
varies along the genome, 
so the next modelling step would be to incorporate this spatial dependance into the model. 
We end up with a random walk $(X_n)_{n\in\N}$
on the hypercube $\{\,0,1\,\}^N$, for which the transition probabilities only depend on the number of differences between two states.
Let 
$\tau_0$ be
the hitting time of the master sequence $0\cdots 0$, i.e.,
$$\tau_0\,=\,\inf\,\big\{\,n\geq 1:X_n=0\cdots 0\,\big\}\,.$$
The goal is then to compute the expected value of $\tau_0$. 

\noindent
{\bf Lumping.}
The smart way to analyze this random walk is to lump together the states of the hypercube into 
Hamming classes.
The Hamming class number $i$ consists of the points which have $i$
digits equal to $1$ and $N-i$ equal to $0$.
So we define a new process $(Y_n)_{n\in\N}$ by setting
\begin{equation*}
Y_n\,=\,\text{number of digits of $X_n$ equal to $1$}
	\,.
\end{equation*}
We obtain a Markov chain with state space $\{\,0,\dots,N\,\}$,
and we shall provide explicit formulas for its mean passage times.

\noindent
{\bf Discovering and recovering time.}
For $0\leq j\leq N$, we define
the hitting time of the Hamming class $j$ by
$$\tau_j\,=\,\inf\,\big\{\,n\geq 1:Y_n=j\,\big\}\,.$$
The three theorems 
below give exact formulas for the expected value of $\tau_j$, when starting
from $0$, $N$, or $j$. These formulas are surprisingly simple and come out from tricky computations.
The discovering time of the master sequence is bounded from above by 
the traversal time, which is the time
needed to reach the class $0$ starting from the class $N$. 
This corresponds to the situation where we start with a string containing only ones, and 
we wait until we see a string containing only zeroes.
\begin{theorem}
\label{trav}
The mean traversal time is given by
$$E\big(\tau_0\,\big|\,Y_0=N\big)\,=\,
\sum_{k=1}^N
\displaystyle{N \choose {k}}
	\frac{1-(-1)^k}{1-(1-2p)^k}\,.$$
\end{theorem}
The recovering time of the master sequence corresponds to the 
returning time to the class $0$ when starting away from a genotype different from $0\cdots 0$.
This time is bounded from below by the mean return time to $0$ starting from $0$, which we compute next.
\begin{theorem}
\label{return}
The mean returning time to the class $0$ is given by
$$E\big(\tau_0\,\big|\,Y_0=0\big)\,=\,2^N
\,.$$
\end{theorem}
Of course, the formula of theorem~\ref{return} is also a straightforward consequence of the classical
result expressing the invariant probability measure of a Markov chain in terms of mean recurrence times.
However, it does not seem that the other formulas presented in theorems~\ref{trav} or~\ref{greturn}
are easy consequences of more general results.
We compute next a beautiful formula for the returning time to the class~$j$ when starting away from a genotype of the same class.
\begin{theorem}
\label{greturn}
For $1\leq j\leq N$,
the mean returning time to the class $j$ is 
$$E\big(\tau_j\,\big|\,Y_0=j\big)\,=\,
\frac{
\displaystyle
2^N}{\displaystyle{N \choose {j}}
}\,.$$
\end{theorem}
	From 
	theorems~\ref{trav} and~\ref{return}, it is easy to infer an estimate on the mean returning time
	of $0$ which is uniform with respect to the starting point. Indeed, a standard coupling
	argument yields that, for any starting string $x_0$,
$$E\big(\tau_0\,\big|\,Y_0=0\big)\,\leq\,
E\big(\tau_0\,\big|\,X_0=x_0\big)\,\leq\,
E\big(\tau_0\,\big|\,Y_0=N\big)\,.$$
Taking into account theorems~\ref{trav} and~\ref{return}, we conclude that
$$2^N\,\leq\,
E\big(\tau_0\,\big|\,X_0=x_0\big)\,\leq\,
\sum_{k=1}^N
\displaystyle{N \choose {k}}
	\frac{1-(-1)^k}{1-(1-2p)^k}\,\leq\,
	\frac{2^N}{p}\,.$$
	These inequalities show that the discovering and the recovering times of the master
	sequence are of order $2^N$.  
	It turns out that these results are akin to those related to an old classical model, 
	the Ehrenfest model, that we describe briefly.

\noindent
{\bf The Ehrenfest model.}
Let us consider $N$ balls and two boxes.
Initially, all the balls are in the first box. At each time step, one ball is selected
at random and is moved from its current box to the other box. The central question is then the following:
\smallskip

\centerline{On average, how long will it take 
to return to the initial state?}
\smallskip

\noindent
In 1947, 
Mark Kac gave a simple answer to this question in a celebrated paper \cite{KAC}.
He considered the evolution of the number of balls in the first box. This process is a Markov chain on
$\{\,0,\cdots,N\,\}$, which is quite different from our process $(Y_n)_{n\in\N}$. For instance,
its increments are either $-1,0$ or $+1$.
Mark Kac showed that,
starting from $0$, 
the average time for the Ehrenfest process to return to its initial state is equal to
$2^N$, which is the analog of theorem~\ref{return}.
He showed also that,
when starting from the state $j$, 
the average time until return to $j$ is equal to
	${ 2^N}/{N \choose {j}}$. Theorem~\ref{greturn} gives the analogous result for our model of mutations.
%
	\smallskip

	\noindent
{\bf A glimpse of potential theory}.
We attack here the general case, i.e., we look for a formula for
the mean returning time to the class $j$,
when the process starts from the class $i$.
We start from the formula obtained in theorem~\ref{trav}, for the specific case where $i=N$ and $j=0$.
We expand in a geometric series the denominator and we get
$$E\big(\tau_0\,\big|\,Y_0=N\big)\,=\,
\sum_{k=1}^N
\sum_{n\geq 0}
\displaystyle{N \choose {k}}
 \Big( (1-2p)^{nk}
- (-1)^k (1-2p)^{nk}\Big)
 \,.$$
We exchange the order of the summations and, using the formulas~\eqref{gutq}
and~\eqref{gutr}, we obtain
$$
E\big(\tau_0\,\big|\,Y_0=N\big)\,=\, 2^N \sum_{n \geq 0}\Big(P_0\big(Y_n=0\big)-P_N\big(Y_n=0\big)\Big)\,.
$$
From theorem~\ref{return}, we have also that
$E\big(\tau_0\,\big|\,Y_0=0\big)=2^N$. 
The above display formula is actually a particular case of a more general
identity valid for a large class of Markov chains, that we state in the next theorem.
\begin{theorem}
\label{general}
For any $i,j$ in $\{\,0,\dots,N\,\}$, we have
$$E\big(\tau_j\,\big|\,Y_0=i\big)\,=\,E\big(\tau_j\,\big|\,Y_0=j\big)
	\bigg(\sum_{n \geq 0}\Big(P_j\big(Y_n=j\big)-P_i\big(Y_n=j\big)\Big)\bigg)\,.
$$
\end{theorem}
In the case of our specific model, we have further a formula for each term in the sum
(see formula~\eqref{genfor}). This way, we get an exact formula for
$E\big(\tau_j\,\big|\,Y_0=i\big)$, which we present in the next theorem.
\begin{theorem}
\label{greturngen}
For $1\leq j\leq N$,
the mean returning time to the class $j$ starting from class $i$ is given by
\begin{multline}
E\big(\tau_j\,\big|\,Y_0=i\big)\,=\\
\frac{1}
{\displaystyle{N \choose {j}}
}
\sum_{n\geq 0} \sum_{k=0}^N 
\Bigg(
{j \choose k} {N-j \choose k} - {i \choose j-k}{N-i \choose k} \Big(\frac{1-(1-2p)^n}{1+(1-2p)^n}\Big)^{i-j}
\Bigg) \\
\Big(1-(1-2p)^n\Big)^{2k} \Big(1+(1-2p)^n\Big)^{N-2k}
\,.
\end{multline}
\end{theorem}
{\bf Strategy of the proof}.
We will employ the same method that Mark Kac used for the Ehrenfest model.
We shall try to compute
$E\big(\tau_j\,\big|\,Y_0=i\big)$ for $0\leq i,j\leq N$.
These computations are tricky. 
We will in fact compute the generating functions
of the event $\{\,Y_n=j\,\}$ and of the random variable $\tau_j$, and we will
relate them through a functional equation that we derive in the next section.
The mean passage times are equal to the left derivative at $1$ of the generating
function of $\tau_j$. We compute these derivatives by performing a local expansion of the
functions around~$1$.
Finally, in the last section, we prove a general formula for 
$E\big(\tau_j\,\big|\,Y_0=i\big)$ for $0\leq i,j\leq N$, which is based on the potential
theory for Markov chains.
\section{A classical probabilistic identity}
\label{clid}

%
Throughout the computations, we shall denote by $P_i$ the 
probability conditioned on the event that $Y_0=i$. 
Let us fix $0\leq i,j\leq N$. 
For $n\geq 1$, 
we compute $P_i(Y_n=j)$ by decomposing the
event $\{\,Y_n=j\,\}$ according to the hitting time $\tau_j$, as follows:
\begin{align*}
P_i(Y_n=j)\,&=\,
\sum_{k=1}^n
P_i(Y_n=j,\,\tau_j=k)\cr
\,&=\,
\sum_{k=1}^n
P_i\big(Y_1\neq j,\dots,Y_{k-1}\neq j, Y_k=j,Y_n=j\big)\,.
\end{align*}
We perform a conditioning in the probability in the sum and we get
\begin{multline*}
P_i(Y_n=j)\,=\,
\smash{\sum_{k=1}^n}
P_i\big(Y_1\neq j,\dots,Y_{k-1}\neq j, Y_k=j\big)\cr
\qquad\qquad\times
 P_i\big(Y_{n}=j\,\big|\,
Y_1\neq j,\dots,Y_{k-1}\neq j, Y_k=j
\big)\,.
\end{multline*}
The constitutive property of a Markov chain is that the past influences the future only
through the present state. This is rigorously formalized in the Markov property, which
yields that
$$ P_i\big(Y_{n}=j\,\big|\,
Y_1\neq j,\dots,Y_{k-1}\neq j, Y_k=j
\big)
\,=\,
P_i\big(Y_{n}=j\,\big|\, Y_k=j\big)\,.$$
Since in addition the Markov chain is time homogeneous, we have that
$$P_i\big(Y_{n}=j\,\big|\, Y_k=j\big)
\,=\,
P_j\big(Y_{n-k}=j\big)
\,.$$
Plugging the last two equalities in the sum, we conclude that
\begin{equation}
\label{iden}
P_i(Y_n=j)\,=\,
\sum_{k=1}^n
P_i\big(\tau_j=k\big)P_j\big(Y_{n-k}=j\big)\,.
\end{equation}
To take advantage of this identity, we introduce the generating functions
of the event $\{\,Y_n=j\,\}$ and of the random variable $\tau_j$, i.e.,
we consider the series
\begin{align*}
F_{ij}(z)\,&=\,
\sum_{n\geq 1}
P_i(Y_n=j)\,z^n\,,\cr
G_{ij}(z)\,&=\,
\sum_{n\geq 1}
P_i(\tau_j=n)\,z^n\,.
\end{align*}
Their radius of convergence is at least one.
In equation~\eqref{iden}, we isolate the term corresponding to $k=n$, and we multiply
by $z^n$ to get
\begin{equation*}
P_i(Y_n=j)z^n\,=\,
P_i\big(\tau_j=n\big)z^n\,+\,
\sum_{k=1}^{n-1}
P_i\big(\tau_j=k\big)z^k\,P_j\big(Y_{n-k}=j\big)z^{n-k}\,.
\end{equation*}
Looking at the sum, we recognize the Cauchy product of the two series. Summing this
identity, we conclude that
\begin{equation}
\label{idgen}
F_{ij}(z)\,=\,
G_{ij}(z)\,+\,
F_{jj}(z)\,
G_{ij}(z)\,.
\end{equation}
The strategy is now the following. We try to compute the functions $F_{ij}$.
We use then the above identity to obtain the functions $G_{ij}$.
Finally, the mean passage times can be computed from $G_{ij}$
by taking its left derivative at~$1$:
$$
E\big(\tau_j\,\big|\,Y_0=i\big)\,=\,
\sum_{n\geq 1} nP_i(\tau_j=n)\,=\,
G'_{ij}(1)\,.$$

\section{Single nucleotide dynamics}
We suppose here that $N=1$, i.e., we focus on the dynamics of a single nucleotide. 
In this case, the process
$(X_n)_{n\geq 0}$ is
the Markov chain 
with state space $\{0,1\}$ and transition
matrix
$$M\,=\,\begin{pmatrix}
\displaystyle 1-p & p \cr
	p & 1-p\cr
\end{pmatrix}\,.
$$
The eigenvalues of $M$ are $1$ and 
$1-2p$.
We compute, for $n\geq 1$,
$$M^n\,=\,
\frac{1}{2}
\begin{pmatrix}
	\displaystyle 1+ (1-2p)^n & 1-(1-2p)^n  \cr
	\displaystyle 1- (1-2p)^n & 1+(1-2p)^n  \cr
\end{pmatrix}\,.
$$
Here is a simple illuminating way to realize the dynamics and to
understand the expression of the $n$--th power $M^n$.
Let
$(\varepsilon_n)_{n\geq 1}$ be an i.i.d. sequence of Bernoulli random variables with
parameter $p$. At each time step, we use the variable $\varepsilon_n$ to decide
whether $X_n$ mutates or not. More precisely, we set
$$X_n\,=\,
\begin{cases}
X_{n-1}\quad&\quad\text{if}\quad \varepsilon_n=0\,,\cr
1-X_{n-1}\quad&\quad\text{if}\quad \varepsilon_n=1\,.\cr
\end{cases}
$$
Now, the event $X_n=X_0$ occurs if and only if the total number of
mutations which happened until time $n$ is even, i.e.,
$$
P(X_n=X_0)\,=\,P\big(\varepsilon_1+\cdots+\varepsilon_n\text{ is even}\big)\,.
$$
Let us set
$$S_n\,=\,\varepsilon_1+\cdots+\varepsilon_n\,.
$$
Here is a little trick to compute the probability that $S_n$ is even.
We compute in two different ways the expected value of $(-1)^{S_n}$. Indeed, we have
\begin{align*}
E\big((-1)^{S_n}\big)&\,=\,
\Big(
E\big((-1)^{\varepsilon_1}\big)\Big)^n
\,=\, \big(-p+1-p\big)^n
\,=\, \big(1-2p\big)^n
\cr
&\,=\,
P\big(S_n\text{ is even}\big)\,-\,
P\big(S_n\text{ is odd}\big)\,.
\end{align*}
Obviously, we have
$$
P\big(S_n\text{ is even}\big)\,+\,
P\big(S_n\text{ is odd}\big)\,=\,1\,,$$
therefore we obtain that
$$
	P(X_n=X_0)\,=\,
P\big(S_n\text{ is even}\big)\,=\,
	\frac{1}{2} \Big(1+(1-2p)^n\Big)\,.$$
This way we recover the expression of the diagonal coefficients of $M^n$.
Let us define
\begin{equation}
\label{albe}
p_n\,=\,
	\frac{1}{2} \Big(1+(1-2p)^n\Big)\,.
\end{equation}
From the above computations, 
we conclude the following. 
Conditionally on $X_0=1$, $X_n$ is a 
Bernoulli random variable with parameter $p_n$, i.e.,
$$P\big(X_n=1\,\big|\,X_0=1\big)\,=\,p_n\,,\qquad
P\big(X_n=0\,\big|\,X_0=1\big)\,=\,1-p_n\,.
$$
Similarly, conditionally on $X_0=0$, $X_n$ is a 
Bernoulli random variable with parameter $1-p_n$.
\section{Multiple nucleotides dynamic}
We consider now the case where the number $N$ of nucleotides is larger than one.
In our model, the mutations occur independently at each site. An important consequence
of this structural assumption is that the components of $X_n$,
$(X_n(i),1\leq i\leq N)$, are themselves Markov chains like the one studied in the previous
section, and these Markov chains are moreover independent.
This remark, combined with the results of the previous section, allows to derive 
explicitly the distribution of $Y_n$.
Indeed, suppose that we start from $Y_0=i$.
This means that $i$ digits in $X_0$ are equal to $1$ and $N-i$ to $0$.
At time $n$, in $X_n$, the $i$ digits which were initially equal to $1$ are distributed according
to a Bernoulli law of parameter $p_n$, the others
are distributed according
to a Bernoulli law of parameter $1-p_n$.
The evolution of the nucleotides being independent, these Bernoulli variables are independent, so their sum
is distributed as the sum of two independent Binomial random variables:
$$Y_n\,\sim\,\bino(i,p_n)+
\,\bino(N-i,1-p_n)\,.$$
This yields for instance the following formula:
\begin{multline}
	\label{genfor}
	P_i(Y_n=j)\,=\,
\kern-7pt
\sum_{\tatop{0\leq k\leq i}{ 0\leq j-k\leq N-i}}
\kern-10pt
P\big( \bino(i,p_n)=k\big)
P\big( 
\bino(N-i,1-p_n)
=j-k\big)\cr
	\hfill
	=\kern-7pt
\sum_{\tatop{0\leq k\leq i}{ 0\leq j-k\leq N-i}}
\kern-7pt
{i \choose k}
	{N-i \choose j-k}
	(1-p_n)^{i+j-2k}(p_n)^{N-i-j+2k}
\,.
\end{multline}
This formula is quite complicated. Yet it becomes particularly simple in the cases
where $i$ or $j$ is equal to $0$ or $N$. Indeed, we have, for $0\leq i\leq N$,
\begin{align*}
	P_i(Y_n=0)\,&=\,
(1-p_n)^i (p_n)^{N-i}\,,\cr
\qquad
	P_i(Y_n=N)\,&=\,
(p_n)^i (1-p_n)^{N-i}\,,
\end{align*}
and
for $0\leq j\leq N$,
\begin{align}
	P_0(Y_n=j)\,&=\,
{N \choose j}
	(1-p_n)^{j}(p_n)^{N-j}
	\label{trag}
	\,,\\
	\label{trai}
\qquad
	P_N(Y_n=j)\,&=\,
{N \choose j}(p_n)^j(1-p_n)^{N-j}
	\,.
\end{align}
For once, surprisingly enough, these two cases are also the most relevant for genetic applications,
so we treat them first. We will indeed compare these extreme cases to the general chain and deduce an estimation on the discovering and returning time.
\section{Proofs of theorems~\ref{trav} and~\ref{return}.} 
	This section is devoted to the completion of the proof of theorems~\ref{trav} and~\ref{return}. 
	We shall implement the strategy explained at the end of section~\ref{clid}.
Our first goal is to compute the generating function
$$F_{N0}(z)\,=\,
\sum_{n\geq 1}
P_N(Y_n=0)\,z^n\,.$$
From formulas~\eqref{trai} and~\eqref{albe}, we have
\begin{equation*}
P_N(Y_n=0)\,=\,
	(1-p_n)^{N}
	\,=\,
	\Big(
	\frac{1-
	(1-2p)^n
	}{2} 
	\Big)^N
	\,.
\end{equation*}
	We use the binomial expansion to develop the $N$--th power in order to compute
	the generating function $F_{N0}$ as a sum of geometric series:
\begin{equation}
	\label{gutq}
P_N(Y_n=0)\,=\,
 \frac{1}{2^N}\sum_{k=0}^N \binom{N}{k}(-1)^k (1-2p)^{nk}.
\end{equation}
Notice that $P_N\big(Y_0=0\big)=0$. For convenience, we start the sum defining $F_{N0}$ at $n=0$
and we obtain a finite number of geometric series:
\begin{align*}
	F_{N0}(z)&\,=\,
\sum_{n\geq 0}
P_N(Y_n=0)\,z^n\\
&\,=\,\sum_{n \geq 0} 
 \frac{1}{2^N}\sum_{k=0}^N \binom{N}{k}(-1)^k (1-2p)^{nk}
	z^n\\
&\,=\, \frac{1}{2^N}\sum_{k=0}^N \binom{N}{k}
	\frac{ (-1)^k }{1-(1-2p)^k z}\,. 
\end{align*}
Our next goal is to compute the generating function
$$F_{00}(z)\,=\,
\sum_{n\geq 1}
P_0(Y_n=0)\,z^n\,.$$
From formulas~\eqref{trag} and~\eqref{albe}, we have, after binomial expansion:
\begin{align}
	P_0(Y_n=0)&\,=\,
	(p_n)^{N}
	\,=\,
	\Big(
	\frac{1+
	(1-2p)^n
	}{2} 
	\Big)^N\cr
	&\,=\,
 \frac{1}{2^N}\sum_{k=0}^N \binom{N}{k} (1-2p)^{nk}
	\,.
	\label{gutr}
\end{align}
This time, we have
	$P_0\big(Y_0=0\big)=1$. Adding this term to $F_{00}$, we get again nice geometric series:
\begin{align*}
	1+F_{00}(z)&\,=\,
\sum_{n\geq 0}
P_0(Y_n=0)\,z^n\\
&\,=\,\sum_{n \geq 0} 
 \frac{1}{2^N}\sum_{k=0}^N \binom{N}{k} (1-2p)^{nk}
	z^n\\
&\,=\, \frac{1}{2^N}\sum_{k=0}^N \binom{N}{k}
	\frac{ 1 }{1-(1-2p)^k z}\,. 
\end{align*}
For $0\leq k\leq N$, 
we introduce the auxiliary functions
$$
\phi_k(z) = \binom{N}{k} \frac{1}{1-(1-2p)^k z}\,,
$$
and we rewrite the expressions of 
$F_{N0}$ and $1+F_{00}$ as
\begin{align}
	\label{foo}
		F_{N0}(z)\,=\,
 \frac{1}{2^N}\sum_{k=0}^N  (-1)^k\phi_k(z)
 \,,\cr
	1+F_{00}(z)\,=\,
 \frac{1}{2^N}\sum_{k=0}^N  \phi_k(z)
 \,.
\end{align}
We have computed $F_{N0}$ and $1+F_{00}$. From the probabilistic identity~\eqref{idgen}, 
we obtain
$$
G_{N0}(z)\,=\,
\frac{F_{N0}(z)}{1+F_{00}(z)}\,.$$
Remember that our ultimate goal is to compute the left derivative of $G_{N0}$ at $1$.
The functions $\phi_k$ are regular around $1$, except the first one, $\phi_0$, indeed,
$$\phi_0(z)\,=\,\frac{1}{1-z}\,.$$
To get $G'_{N0}(1)$, we perform a local expansion of $G_{N0}$ around $1$, as follows:
\begin{align*}
G_{N0}(z) 
	&\,=\,
\frac{\displaystyle \frac{1}{1-z}+ \sum_{k=1}^N  (-1)^k\phi_k(z)}
{\displaystyle \frac{1}{1-z}+ \sum_{k=1}^N  \phi_k(z)}\\
	&\,=\,
	\displaystyle {1}+({1-z}) \sum_{k=1}^N  \big((-1)^k-1\big)\phi_k(z)+o(z-1)\,.
\end{align*}
This expansion readily yields the value of the left derivative of $G_{N0}$ at $1$:
$$G'_{N0}(1) \,=\,
	 \sum_{k=1}^N  \big(1-(-1)^k\big)\phi_k(1)\,.
$$
Replacing $\phi_k(1)$ by its value, we obtain the formula stated in theorem~\ref{trav}.
We proceed similarly to prove
theorem~\ref{return}. In fact, we have to compute $G'_{00}(1)$, and the 
probabilistic identity~\eqref{idgen} yields
$$
G_{00}(z)
\,=\,
\frac{F_{00}(z)}{1+F_{00}(z)}
\,=\,1-
\frac{1}{1+F_{00}(z)}\,.$$
We have already computed $1+F_{00}(z)$ in formula~\eqref{foo}. We use this expression
and we expand around $z=1$:
\begin{equation*}
G_{00}(z) \,=\,1- \frac{\displaystyle 2^N}{
	\displaystyle\sum_{k=0}^N  \phi_k(z)}
\,=\,1- 2^N (1-z)+o(1-z)\,.
\end{equation*}
This expansion shows that
$G'_{00}(1)=2^N$.

\section{Proof of theorem~\ref{greturn}}
We shall finally prove the analog of Kac theorem on the mean returning time to the class $j$,
when the process starts from the class $j$.
We write the formula~\eqref{genfor} with $i=j$, we reindex the sum by setting
$\ell=j-k$
and we perform the two binomial expansions:
\begin{multline*}
	P_j(Y_n=j)\,=\,
\cr
	\hfill
	\kern-7pt
\sum_{\tatop{0\leq k\leq j}{ 0\leq j-k\leq N-j}}
\kern-7pt
{j \choose k}
	{N-j \choose j-k}
	\Big(
	\frac{ 1-(1-2p)^n }{2} 
	\Big)^{2j-2k}
	\Big(
	\frac{ 1+(1-2p)^n }{2} 
	\Big)^{N-2j+2k}
\cr
	\hfill
	\kern-7pt
\,=\,
	\sum_{\ell=0}^{ j\land{(N-j)}}
{j \choose \ell}
	{N-j \choose \ell}
	\Big(
	\frac{ 1-(1-2p)^n }{2} 
	\Big)^{2\ell}
	\Big(
	\frac{ 1+(1-2p)^n }{2} 
	\Big)^{N-2\ell}
	\hfill
	\cr
\,=\,
	\kern-7pt
	\sum_{\ell=0}^{ j\land{(N-j)}}
	\kern-5pt
	\frac{ 1 }{2^N} 
{j \choose \ell}
	{N-j \choose \ell}
	\sum_{\alpha=0}^{2\ell} 
\sum_{\beta=0}^{N-2\ell} 
	\binom{2\ell}{\alpha} 
\binom{N-2\ell}{\beta} 
	(-1)^{\alpha}(1-2p)^{(\alpha+\beta) n} 
\,.
\end{multline*}
	For $n=0$, we have
	$P_j(Y_0=j)=1$, therefore, after a geometric summation, we get
	$$\displaylines{
	1+F_{jj}(z)\,=\,
\sum_{n\geq 0}
P_j(Y_n=j)\,z^n\hfill\cr
	\,=\,
	\kern-7pt
	\sum_{\ell=0}^{ j\land{(N-j)}}
	\kern-5pt
	\frac{ 1 }{2^N} 
{j \choose \ell}
	{N-j \choose \ell}
	\sum_{\alpha=0}^{2\ell} 
\sum_{\beta=0}^{N-2\ell} 
	\binom{2\ell}{\alpha} 
\binom{N-2\ell}{\beta} 
	\frac{(-1)^\alpha}{1-(1-2p)^{\alpha +\beta } z}
	\,.}$$
We expand this function around $z=1$ and we get
\begin{align*}
	1+F_{jj}(z)&\,=\,
	\kern-7pt
	\sum_{\ell=0}^{ j\land{(N-j)}}
	\kern-5pt
	\frac{ 1 }{2^N} 
{j \choose \ell}
	{N-j \choose \ell}
	\frac{1}{1-z}+O(1)\cr
	&\,=\,
	\frac{ 1 }{2^N} 
	{N \choose j}
	\frac{1}{1-z}+O(1)\,,
\end{align*}
thanks to the combinatorial identity stated in the next lemma.
\begin{lemma}\label{binom}
For~$0\leq j\leq N$, we have
$$
	\sum_{\ell=0}^{ j\land{(N-j)}}
	\kern-5pt
{j \choose \ell}
	{N-j \choose \ell}\,=\,
	{N \choose j}\,.
$$
\end{lemma}
\begin{proof}
	Let us fix $j$ in $\{\,0,\dots,N\,\}$, and let us consider a set $E$
	having cardinality $N$.
	We fix also a subset $A$ of $E$ having $j$ elements.
	We classify the subsets of $E$ having cardinality $j$ according
	to the cardinality of their intersection with $A$ and we readily obtain
	the formula of the lemma.
\end{proof}

\noindent
From the probabilistic identity~\eqref{idgen}, 
we have
$$
G_{jj}(z)
\,=\,1-
\frac{1}{1+F_{jj}(z)}\,,$$
thus
$G_{jj}(z)$ admits the following expansion around $z=1$:
$$
G_{jj}(z)\,=\,
	1+
	\frac{2^N}{\displaystyle {N \choose j}}
	({z-1})+o(z-1)\,.
$$
From this expansion, we infer that 
$$E\big(\tau_j\,\big|\,Y_0=j\big)\,=\,
G'_{jj}(1)\,=\,
	\frac{2^N}{\displaystyle {N \choose j}}
	$$
	and this concludes the proof of theorem~\ref{greturn}.
\section{ Proof of theorem~\ref{general}}
It has been observed a long time ago that the equations defining the invariant measure, or the 
returning time to a set for a random walk, or more generally a Markov chain, are formally
equivalent to the equations arising in potential theory, if one interprets the transition probabilities
as conductances (see the very nice book \cite{doyle}). 
In fact, the formula presented in theorem~\ref{general} takes its roots in potential theory \cite{neveu}.
Let us denote by $P$ the transition matrix of the process
$(Y_n)_{n\geq 0}$, defined by
$$\forall i,j\in\{\,0,\dots,N\,\}\quad\forall n\geq 0
\qquad
P(i,j)\,=\,P(Y_{n+1}=j\,|\,Y_n=i)\,.$$
The arguments presented below are in fact valid for a general class of Markov chains with finite
state space.
For instance, it suffices that $P$, or one of its powers, has all its entries positive. In this 
situation, the classical ergodic theorem for Markov chains ensures the existence and uniqueness
of an invariant probability measure and the following convergence holds:
\begin{equation}
	\label{ergodic}
\forall i,j\in\{\,0,\dots,N\,\}
\qquad
\lim_{n\to\infty}
P^n(i,j)\,=\,\frac{1}{ E\big(\tau_j\,\big|\,Y_0=j\big) }\,.
\end{equation}
From now on, we fix $j$ in 
$\{\,0,\dots,N\,\}$ and we try to compute
$E\big(\tau_j\,\big|\,Y_0=i\big) $. The idea is to study
the behavior of the Markov chain until the time~$\tau_j$. To do so, we introduce
the companion matrix $G$ defined by
$$\forall i,k\in\{\,0,\dots,N\,\}\qquad
G(i,k)\, =\, E_i\Big(\sum_{n=0}^{\tau_j-1}1_{\{Y_n=k\}}\Big)\, .  $$
The matrix $G$ is called the potential matrix associated with the restriction 
of $P$ to 
$\{\,0,\dots,N\,\}\setminus \{j\}$. 
The quantity $G(i,k)$ represents the average number of visits of the state $k$ before reaching the state $j$
when starting from $i$. We introduce also the matrix $H$ given by
$$\forall i,k\in\{\,0,\dots,N\,\}\qquad
H(i,k)\, =\, 
\begin{cases}
1&\quad\text{if}\quad k=j\,\cr
0&\quad\text{if}\quad k\neq j\,\cr
\end{cases}
\, .  $$
The matrix $H$ describes the distribution of the exit point from the set 
$\{\,0,\dots,N\,\}\setminus \{\,j\,\}$. In our case, it is necessarily the Dirac mass on $j$, yet
in the general case, the matrix $H$ is more involved!
%
%
%
The three matrices $P,G,H$ are linked through a simple identity.
\begin{lemma}
	\label{lempot}
	Denoting by $I$ the identity matrix, 
we have
$$GP=H+G-I\, .
$$
\end{lemma}
\begin{proof}
The matrix $G$ encodes the behaviour of the process until it hits $j$. Multiplying on the right $G$ 
by the transition matrix $P$ amounts to perform one further step of the process. 
This step might either stay inside 
$\{\,0,\dots,N\,\}\setminus \{\,j\,\}$, in which case we recover the matrix $G-I$, or it might land in
$j$, and this is where the matrix $H$ enters the game.
Let us make this argument rigorous.
We have to check that
$$\forall i,k
\in\{\,0,\dots,N\,\}\qquad
	GP(i,k)\,=\,H(i,k)+G(i,k)-I(i,k)\, .
	$$
For $i,k
\in\{\,0,\dots,N\,\}$, we compute
\begin{align*}
GP(i,k) \, =\,
 & \sum_{0\leq\ell\leq N}\, G(i,\ell)\,P(\ell,k)
	\\
\,=\, & 
  \sum_{0\leq\ell\leq N}
	E_i\Big(\sum_{n\geq 0 }1_{\{\tau_j>n\}}1_{\{Y_n=\ell\}}\Big)
	\,P(\ell,k) \\
\,=\, & 
  \sum_{0\leq\ell\leq N}
	\sum_{n\geq 0 }
	P_i\Big(
	\tau_j>n,\,Y_n=\ell\Big)
	\,P\big(Y_{n+1}=k\,\big|\,Y_n=\ell\big) \\
\,=\, & \sum_{n\geq 0 }
  \sum_{0\leq\ell\leq N}
	P_i\Big(
	\tau_j>n,\,Y_n=\ell,\,
	Y_{n+1}=k \Big)
	\\
\,=\, & 
	P(i,j)+
	\sum_{n\geq 1 }
	P_i\Big(
	\tau_j>n,\,
	Y_{n+1}=k \Big)\,.
\end{align*}
We consider now two cases.
If $k=j$, the formula becomes
$$GP(i,j) \,=\,
	\sum_{n\geq 0 }
	P_i\big(
	\tau_j=n+1
	\big)
\,=\,  1\,=\,
	H(i,j)+G(i,j)-I(i,j)\, .
	$$
If $k\neq j$, the formula becomes
\begin{align*}
	GP(i,k) &\,=\,
	P(i,k)+
	\sum_{n\geq 1 }
	P_i\Big(
	\tau_j>n+1,\,
	Y_{n+1}=k \Big)\\
	 &\,=\,
	G(i,k)-I(i,k)	
	\,.
\end{align*}
This ends the proof, since $H(i,k)=0$ in this case.
\end{proof}

\noindent
We complete finally the proof of theorem \ref{general}.
We multiply the formula of lemma~\ref{lempot} by $P^n$ and we sum from $0$ to $m$ to obtain
$$
G -G P^{m+1}\,=\,\sum_{n=0}^m \big( P^n -H P^n\big) \,.
$$ 
We focus on the coefficients $(i,j)$ of the matrices and we send $m$ to $\infty$:
$$
\lim_{m\rightarrow \infty} 
\big(G(i,j)
-GP^m(i,j)
\big)\,=\,\sum_{n\geq0} \big( P^n(i,j) - P^n(j,j)\big).
$$
Now $G(i,j)=0$ and
from the convergence~\eqref{ergodic}, we have
\begin{equation*}
\lim_{m\rightarrow \infty}GP^m(i,j)\,=\,
	\Big(\sum_{k=0}^N G(i,k)\Big)\times
	\frac{1}{E\big(\tau_j\,\big|\,Y_0=j\big)}\, .
\end{equation*}
Noticing that
	$$\sum_{k=0}^N G(i,k)\,=\,
	{E\big(\tau_j\,\big|\,Y_0=i\big)}\,,$$
	and putting together the previous identities, we obtain the formula stated in
	theorem~\ref{general}. It then suffices to replace the probabilities with their expression to get theorem \ref{greturngen}.

\bibliographystyle{plain}

\thispagestyle{empty}

\end{document}